\documentclass[12pt]{article}
\usepackage[margin =1in]{geometry}
\usepackage{amssymb,amsthm}
\usepackage{amsmath}
\usepackage{enumerate}
\usepackage{hyperref}
\usepackage{cite}

\usepackage{color}

\usepackage{macros}

\numberwithin{equation}{section}
\newcommand{\muh}{\mu_h}
\newcommand{\muf}{\mu_f}
\newcommand{\muc}{\mu_c}
\newcommand{\lh}{L_h}
\newcommand{\lj}{L_{\nabla c}}
\newcommand{\lf}{L_{f}}
\newcommand{\lc}{L_c}
\newcommand{\gnp}{\texttt{GNP}}
\newcommand{\rgnp}{\texttt{RGNP}}
\newcommand{\scaledsm}{\texttt{ScaledSM}}

\newcommand{\polyak}{\texttt{Polyak}}
\newcommand{\rpolyak}{\texttt{RPolyak}}

\newcommand{\closedball}{\bar{B}}
\newcommand{\multo}{\rightrightarrows}

\usepackage{algorithm,algpseudocode}

\begin{document}

	\title{A linearly convergent Gauss-Newton subgradient method for ill-conditioned problems}


	\author{Damek Davis\thanks{School of Operations Research and Information Engineering, Cornell University,
Ithaca, NY 14850, USA;
\texttt{damekdavis.com}. Research of Davis supported by an Alfred P. Sloan research fellowship and NSF DMS award 2047637. Research was completed while Davis was visiting the program on computational microscopy at the Institute of Pure and Applied Mathematics (IPAM) in Fall 2022.}
\and
Tao Jiang\thanks{School of Operations Research and Information Engineering, Cornell University,
Ithaca, NY 14850, USA; {\tt tj293@cornell.edu}.}}

	\date{}
	\maketitle

\begin{abstract}
We analyze a preconditioned subgradient method for optimizing composite functions $h \circ c$, where $h$ is a locally Lipschitz function and $c$ is a smooth nonlinear mapping. We prove that when $c$ satisfies a constant rank property and $h$ is semismooth and sharp on the image of $c$, the method converges linearly. In contrast to standard subgradient methods, its oracle complexity is invariant under reparameterizations of $c$.
\end{abstract}
\section{Introduction}

In this work we develop a subgradient method for the composite optimization problem
\begin{align}\label{eq:composite}
\minimize_{x\in \RR^d}\, f(x) := h(c(x)),
\end{align}
where $h$ is a Lipschitz penalty function and $c $ is a smooth nonlinear mapping. 
We are specifically interested in a method whose iterates linearly converge to the set of solutions $\cX_\ast$.
Prior work on such methods focuses on the setting where $h$ is convex and the full composition $f = h\circ c$ is \emph{sharp and Lipschitz},
meaning there exists $\muf, \lf > 0$ such that 
\begin{align}\label{sharpandlipschitz}
\muf\dist(x, \cX_\star) \leq h(c(x)) - h^\ast \leq \lf \dist(x, \cX_\ast) \qquad\text{ for all $x$ near $\cX_\star$.}
\end{align}
For example, in~\cite{davis2018subgradientWC} it was shown that the so-called subgradient method with Polyak stepsize $(\polyak)$~\cite{polyak1969minimization} finds a point $\varepsilon$ close to $\cX_\ast$ after 
$$
O\left((\lf/\muf)^2\log(1/\varepsilon)\right)
$$
iterations, when properly initialized.
While useful, a drawback of $\polyak$ is that the convergence behavior depends on the ratio $\lf/\muf$, a ``condition number" for the composition $h \circ c$.
This is undesirable since reparameterizations of the domain of $c$ can lead to equivalent, but more poorly conditioned problems.
Motivated by this, we seek to design subgradient methods whose performance depends not on the particular parameterization of $c$, but only on the conditioning of $h$ along the image of $c$.
To achieve this we will accept a higher per-iteration cost, due to solving a linear system, which in some cases has favorable structure.

To describe the assumptions, algorithm, and result, let us suppose for the moment that $h$ is convex and fix a minimizer $x_\star$ of $f$.
We make two assumptions. 
The first assumption is a parameterization invariant generalization of~\eqref{sharpandlipschitz}: we assume $h$ is \emph{sharp and Lipschitz on the image of $c$}, meaning there exists $\muh, \lh > 0$ such that
$$
\muh\dist(c(x), \cZ_\star) \leq h(c(x)) - h^\ast \leq \lh \dist(c(x), \cZ_\ast), \qquad \text{for all $x$ near $x_\star$.}
$$
Here, $h^\ast := \inf_x h(c(x))$ denotes the optimal value and $\cZ_\ast$ denotes of the images $c(x_\star')$ for minimizers $x_\star'$ near $x_\star$.
The second assumption ensures that image of $c$ is sufficiently regular: we assume that 
$$
\nabla c(x)  \text{ has constant rank near $x_\star$}.
$$
A well-known consequence of this property (the ``constant rank theorem"~\cite[Chapter 7]{lee2013smooth}) is that the image of a small neighborhood of $x_\star$ under the mapping $c$ is a smooth manifold. 

Before describing the algorithm, we illustrate these assumptions on the low-rank matrix sensing problem, where $\polyak$ and the condition~\eqref{sharpandlipschitz} have been extensively analyzed.
In this problem, one seeks to recover an unknown symmetric rank-$r$ matrix $M_\star := X_\star X_\star^T$ from known linear measurements $b = \cA(M_\star)$. 
A common formulation of this problem is $\ell_1$ penalization of the residuals: 
$$
\min_{X \in \RR^{d\times r}} \|\cA(XX^T) - b\|_1.
$$ 
For certain choices of $\cA$, this formulation -- with $c(X) = XX^T$ and $h(M) = \|\cA(M) - b\|_1$ -- has been shown to be sharp and Lipschitz~\cite{charisopoulos2021low}.
Moreover, as long as $X_\star$ is full rank, the Jacobian $\nabla c(X)$ has rank $r(r+1)/2$ near $X_\star$ (see Lemma~\ref{lem:bmconstantrank}).
Unfortunately, the condition number of $h\circ c$ is proportional to the squared condition number of $X_\ast$: $\lf/\muh  \sim \kappa^2(X_\star)$~\cite{charisopoulos2021low}.
Therefore, we expect $\polyak$ to be effective for well-conditioned problems, but to slow down when $X_\star$ is poorly conditioned.
Indeed, Figure~\ref{fig:qs_exp1_GNP_SM} confirms the negative effect of $\kappa:= \kappa(X_\star)$ on the performance of the standard subgradient method with Polyak-type stepsize. Here, $d = 1000$ and $r=5$; see Section~\ref{sec:conditionn2} for more details.

\begin{figure}[H]
    \centering
    \begin{minipage}{0.47\textwidth}
        \centering
        \includegraphics[width=\linewidth]{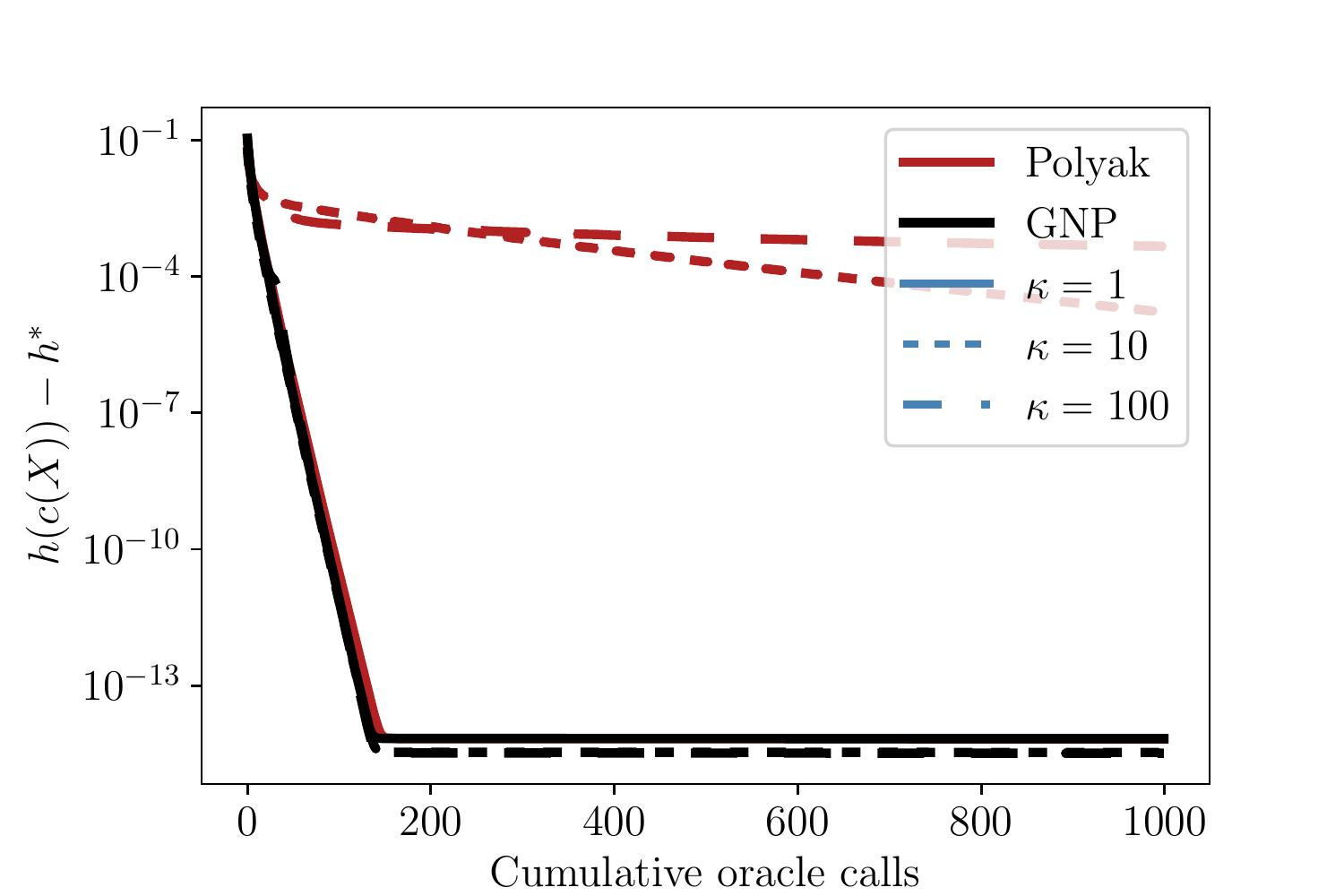}
    \end{minipage}%
    \begin{minipage}{0.47\textwidth}
        \centering
        \includegraphics[width=\linewidth]{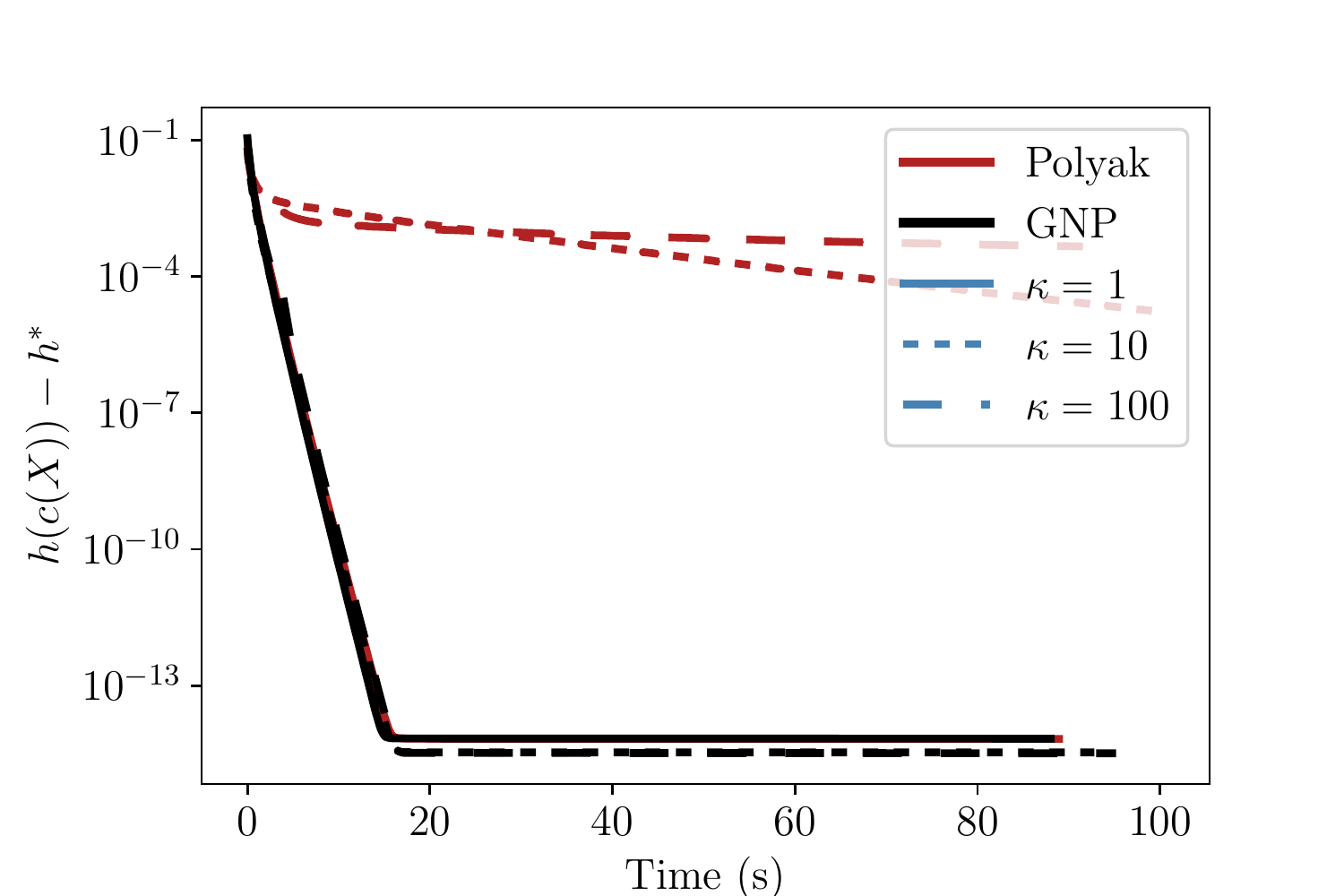}
    \end{minipage}
		\caption{Objective errors for $\polyak$ and $\gnp$ with varying condition number. See Section~\ref{sec:conditionn2} for details.}
		\label{fig:qs_exp1_GNP_SM}
\end{figure}

In contrast to the composition $h\circ c$, the function $h$ is often much better conditioned on the range $c$, in the sense that there exist $\lh, \muh > 0$ such that 
\begin{align}\label{eq:RIP}
\muh \|M - M_\star\|  \leq h(M) \leq \lh \|M - M_\star\| \qquad \text{ for all $M \in \RR^{d\times d}$ of rank $r$.}
\end{align}
This condition is known to hold for several measurement operators $\cA$ and is called as the \emph{restricted isometry property}~\cite{candes2006stable,charisopoulos2021low}.
Moreover, under common modeling assumptions, the ratio $\lh/\muh$ is often a small constant, independent of $\kappa$.
Thus, designing a subgradient method whose performance depends only on $\lh/\muh$ is  desirable.
Indeed, this has already been achieved in the recent work~\cite{tong2021low} by a preconditioning strategy. 
Briefly, the method, called ``Scaled subgradient method" ($\scaledsm$), modifies the subgradient method, scaling the search direction by the inverse Gram matrix of the current iterate.
In~\cite{tong2021low}, the oracle complexity of $\scaledsm$ is shown to be 
$$
O\left((\lh/\muh)^2\log(1/\varepsilon)\right).
$$ 
While there exist generalizations to certain rank $3$ tensor recovery problems~\cite{tong2022scaled}, it is unclear how to generalize the strategy to a wider class of composite problems.

Returning to the general composite problem~\eqref{eq:composite}, the main result of this work is a subgradient method with performance depending only on $\lh/\muh$.
The method we will introduce is motivated by a simple thought experiment: 
suppose one could run the subgradient method on the image of $c$, resulting in a sequence of iterates satisfying:
\begin{align}\label{thoughtexperiment}
c(x_{t+1}) = c(x_t) - \gamma_t v_t \qquad \text{where $v_t \in \partial h(c(x_t))$.}
\end{align}
where $\gamma_t$ denotes a stepsize and $\partial h(c(x_t))$ denotes subdifferential of $h$ at $c(x_t)$.
Then a straightforward argument shows that if (i) $x_t$ stays near $x_\star$ and (ii) we use $\polyak$-type stepsize,
then $x_t$ will satisfy $\dist(c(x_t), \cZ_\ast) \leq \varepsilon$ after at most  $O((\lh/\muh)^2\log(1/\varepsilon))$ iterations.
Unfortunately, it is not possible to implement this method, since $c(x_t) - \gamma_t v_t $ is not  necessarily in the image of $c$.
Instead of an exact solution, one might instead apply a \emph{projected subgradient method}, which replaces~\eqref{thoughtexperiment} with a least squares solution.
This strategy indeed leads to a similar complexity, but has a drawback: the per-iteration complexity could be unreasonably high. 

In this work, we show that one can have the best of both worlds -- dependence on $\lh/\muh$ and reasonable per-iteration complexity -- by solving a linear approximation of~\eqref{thoughtexperiment}.
The method we introduce is called ``Gauss-Newton-Polyak" ($\gnp$), and starting from an initial iterate $x_0$, it satisfies
\begin{align}\label{eq:gnpfirst}
x_{t+1}  = x_t - \gamma_t \nabla c(x_t)^{\dag} v_t \qquad \text{ where } v_t \in \partial h(c(x_t)).
\end{align}
where $\nabla c(x_t)^\dag$ denotes the Moore-Penrose pseudo-inverse of the Jacobian.
We use the name ``Gauss-Newton" because one can interpret the update rule as a subgradient step on $f$, scaled by the pseudo-inverse of the Gauss-Newton preconditioner $\nabla c(x)^T\nabla c(x)$:  
$$
\nabla c(x)^\dag v \in  (\nabla c(x)^T\nabla c(x))^\dag \partial f(x).
$$
We use the name ``Polyak," to refer to our choice of $\gamma_t$, which we will later specify.
To see how $\gnp$ is related to~\eqref{thoughtexperiment}, fix iterate $x_t$ and denote the linearization of $c$ at $x_t$ by
$$
c_t(x) := c(x_t) + \nabla c(x_t)(x - x_t),
$$
where $\nabla c(x_t)$ is the Jacobian of $c$ at $x_t$.
Then one can show that $x_{t+1}$ is solution to the linearization of~\eqref{thoughtexperiment}, which is closest to $x_t$:
\begin{align*}
x_{t+1} = \argmin_{x\in \RR^d} \|x - x_t\|^2 \text{ subject to: } x \in \argmin_{y\in \RR^d}\; \|c_t(y) - (c(x_t) - \gamma_t v_t)\|^2.
\end{align*}
In particular, the iterates $x_t$ will satisfy: 
\begin{align}\label{retraction}
c(x_{t+1}) = c(x_t) - \gamma_t P_{\range(\nabla c(x_t))} v_t + O(\gamma_t^2\|\nabla c(x_t)^\dag v_t\|^2).
\end{align}
Consequently, $x_{t+1}$ is nearly a solution to~\eqref{thoughtexperiment}.

The main contribution of this work is to show that if we use the following modified $\polyak$ stepsize 
$$
\gamma_t := \frac{h(c(x_t)) - h^\ast}{\|P_{\range(\nabla c(x))} v_t\|^2},
$$
in~\eqref{eq:gnpfirst}, then
\begin{quote}
\centering
 $x_t$ will satisfy $\dist(c(x_t), \cZ_\ast) \leq \varepsilon$ after at most  $O((\lh/\muh)^2\log(1/\varepsilon))$ iterations.
\end{quote}
Indeed, Figure~\ref{fig:qs_exp1_GNP_SM} plots the performance, showing that in contrast to $\polyak$, $\gnp$'s performance is independent of the conditioning of the minimizer.
While each step of~$\gnp$ requires solving a linear system, we see that it still outperforms $\polyak$ in terms of time. The reason: the linear system is relatively easy to solve since the Jacobian $\nabla c(X)$ is only rank 15.

While $\gnp$'s performance only depends on $\lh/\muh$, a drawback of the method is that it requires knowledge of $h^\ast$. 
Knowing the precise optimal value is uncommon, 
so a more reasonable assumption is knowing a lower bound.
For example, when vector $b$ in the matrix sensing problem is corrupted by noise, zero is always a lower bound on the optimal value.
Motivated by this, we describe how to apply a recently developed ``restarting" technique proposed in~\cite{hazan2019revisiting}, which only requires a lower bound on $h^\ast$.
The resulting algorithm maintains a similar oracle complexity, but now depends logarithmically on the misspecification of $h^\ast$.

Before turning to the notation, we mention some related work on nonsmooth Gauss-Newton methods,  Riemannian subgradient methods, and approximate projection methods for manifolds.
Prior work on Gauss-Newton methods for nonsmooth optimization exists under the composite sharpness assumption~\eqref{eq:composite}~\cite{burke1995gauss,drusvyatskiy2018error,duchi2019solving}. 
Those methods differ from $\gnp$ in that (a) they converge quadratically and (b) the update step requires solving an auxiliary optimization problem, such as minimizing $w \mapsto h(c(x) + \nabla c(x)w) + \frac{\gamma}{2}\|w\|^2$.
The per-iteration cost of such methods can therefore be substantially higher than $\gnp$.

The results of this work are related to Riemannian subgradient methods for weakly convex problems~\cite{li2021weakly}. Indeed, one may show that $\gnp$ is a Riemannian subgradient method on a manifold $\cM := c(U)$, where $U$ is a small neighborhood of $x_\star$,  with particular retraction $R$. 
Specifically, given $c(x) \in \cM$, the retraction is the mapping $R_{c(x)} \colon c(x) + T_{\cM}(c(x)) \rightarrow \cM$ defined for all $v \in T_{\cM}(c(x))$ by  
$$
c(x) + v \mapsto c( x+ \nabla c(x)^\dag v).
$$ 
This type of retraction appeared in~\cite[Section 4.1.3]{absil2009optimization} and in recent work on approximate projection methods for manifolds~\cite[Section 7]{drusvyatskiy2018inexact}, under the assumption of injectivity.
However, the main result of this paper does not appear to follow from prior work. Indeed,  
we could not existing results on linearly convergent Riemannian subgradient methods with \emph{Polyak stepsize} or the aforementioned ``restart" scheme.

\subsection{Notation and basic constructions.}
Throughout we follow standard convex analysis notation as set out in the monograph~\cite{RW98}.
We let the $\R^d$ denote a $d$-dimensional Euclidean space with the inner product
$\dotp{\cdot,\cdot}$ and  induced norm $\norm{x} = \sqrt{\dotp{x,x}}$. 
Given a matrix $A \in \RR^{m \times d}$, we let $$\sigma_1(A) \geq \sigma_2(A) \geq \ldots \sigma_{\min\{m, d\}}(A) \geq 0$$ denote the singular values of $A$.
We denote the open ball of
radius $\varepsilon>0$ around a point $x\in \R^d$ by the symbol $B_{\varepsilon}(x)$. 
A set-valued mapping $G \colon \RR^d \multo \RR^m$ maps points $x \in \RR^d$ to sets $G(x) \subseteq \RR^m$.
Consider a set $\cX\subseteq\R^d$, the  {\em distance function} and the {\em projection map} are defined by
\begin{equation*}
\dist(x,\cX):=\inf_{y\in \cX} \|y-x\|\qquad \textrm{and}\qquad
P_\cX(x):=\argmin_{y\in \cX} \|y-x\|,
\end{equation*}
respectively. 
We say $\cX$ is locally closed at a point $\bar x$ if there exists a closed neighborhood  $V$ of $\bar x$ such that $V \cap \cX$ is closed. 
In this case,  $P_{\cX}$ is nonempty-valued near $\bar x$.
When the expression ``$x \stackrel{\cX}{\rightarrow} \bar x$" appears in a limit, it indicates that $x \in \cX$ and approaches $\bar x.$
We denote the Local Lipschitz constant of a mapping $F \colon \RR^m \rightarrow \RR^d$ at a point $x_\star$ by 
$$
\lip_F(x_\star) := \limsup_{\substack{x,y \rightarrow x_\star\\ x\neq y}} \frac{\|F(x) - F(y)\|}{\|x - y\|}.
$$

\paragraph {Subdifferentials.}
Consider a locally Lipschitz function $f\colon\R^d\to\R$ and a point $x$. 
The \emph{Clarke subdifferential} is comprised of convex combinations of limits of gradients evaluated at nearby points:
$$
\partial f(x) = \text{conv} \left\{ \lim_{i \rightarrow \infty} \nabla f(x_i) \colon x_i \stackrel{\Omega}{\rightarrow} x\right\},
$$
where $\Omega \subseteq \RR^d$ is the set of points at which $f$ is differentiable (recall Radamacher's theorem). 
If $f$ is $L$-Lipschitz on a neighborhood $U$, then for all $x \in U$ and $v \in \partial f(x)$, we have $\|v\| \leq L$. A point $\bar x$ satisfying $0 \in \partial f(x)$ is said to be critical for $f$.
A function $f$ is called {\em $\rho$-weakly convex} on an open convex set $U$ if the perturbed function $f+\frac{\rho}{2}\|\cdot\|^2$ is convex on $U$. The Clarke subgradients of such functions automatically satisfy the uniform approximation property:
$$f(y)\geq f(x)+\langle v,y-x\rangle-\frac{\rho}{2}\|y-x\|^2\qquad \textrm{for all }x,y\in U, v\in \partial f(x).$$

\paragraph{Manifolds.}
We will need a few basic results about smooth manifolds, which can be found in
the references~\cite{boumal2020introduction,lee2013smooth}. A set $\cM \subseteq
\RR^d$ is called a $C^p$ smooth manifold (with $p \geq 1$) around a point $x$ if there exists a
natural number $m$, an open neighborhood $U$ of $x$,  and a $C^p$ smooth mapping
$F \colon U \rightarrow \RR^m$ such that the Jacobian $\nabla F(x)$ is surjective
and $\cM \cap U = F^{-1}(0)$. The tangent and normal spaces to $\cM$ at
$x \in \cM$ are defined to be $T_\cM(x) = \ker(\nabla F(x))$ and $N_{\cM}(x) =
T_{\cM}(x)^\perp = \range(\nabla F(x)^\ast)$, respectively. 
If $\cM$ is a $C^2$ smooth manifold around a point $\bar x$, then there exists $C > 0$ such that $y - x \in T_{\cM}(x) + C\|y - x\|^2 \closedball$ for all $x, y \in \cM$ near $\bar x$.

\section{Assumptions}

In this section, we more formally state our main assumptions and their consequences.
Setting the stage, we consider the minimization problem
$$
\minimize_{x \in \RR^d}\; h(c(x)), 
$$
where $h \colon \RR^m \rightarrow \RR$ and $c \colon \RR^d \rightarrow \RR^m.$ 
We assume a minimizer $x_\star \in \RR^d$ exists, and denote the optimal value by $h^\ast := \min_x h(c(x))$. Throughout, we assume that $c \colon \RR^d\rightarrow \RR^m$ is a $C^2$ mapping in a neighborhood of $x_\star$ and $h \colon \RR^m \rightarrow \RR$ is a Lipschitz continuous function near $c(x_\star)$. In addition, we make the following assumption on $c$:

\begin{assumption}\label{assumption:constrank}
{\rm
The Jacobian $\nabla c(x)$ is rank $r$ for all $x$ near $x_\star$.}
\end{assumption}
\noindent 
By the \emph{constant rank theorem} in differential topology~\cite[Chapter 7]{lee2013smooth}, a consequence of this assumption and is that there exists a neighborhood $U \subseteq \RR^d$ such that 
$$
\cM := c(U)
$$
 is a $C^2$ smooth manifold around $c(x_\star)$ with tangent space 
 $$
 \cT_{\cM}(c(x)) = \range(\nabla c(x))
 $$ 
for all $x$ sufficiently near $x_\star$. We will use these properties below.

To introduce our remaining assumptions, we define
$$
\cZ_\ast := \argmin_{z \in c(U)} h(z) \qquad\text{ and } \qquad  \cX_\ast := \argmin_{x\in \RR^d} f(x)
$$
Notice that $x_\star \in \cX_\ast$, $c(x_\star) \in \cZ_\ast$, and $\cZ_\star$ is locally closed at $c(x_\star)$.
Besides Lipschitz continuity we make two further assumptions on $h$.
First, we assume it grows at least linearly away from solutions.

\begin{assumption}[Sharpness on image]\label{assume:h}
	{\rm There exists $\muh > 0$ such that the bound holds
$$
h(z) - h^\ast \geq \muh \dist(z, \cZ_\ast) \qquad\text{ for all $z \in c(U)$.}
$$}
\end{assumption}
\noindent As mentioned in introduction, Assumption~\ref{assume:h} holds in several contemporary problems, such as nonconvex formulations of low-rank matrix sensing and completion problems~\cite{charisopoulos2021low}. 
However, Assumption~\ref{assume:h} is somewhat nonstandard, as it is more common to view sharpness as a property of the composition $h\circ c$ as in~\eqref{sharpandlipschitz}. 
The following lemma shows that it is in fact a weaker property.
\begin{lem}
Suppose that $h \circ c$ satisfies~\eqref{sharpandlipschitz}. Then Assumption~\eqref{assume:h} holds.
\end{lem}
\begin{proof}
By local Lipschitz continuity, there exists $\lc > 0$ such that for $x$ near $x_\star$,
$$\frac{\muf}{\lc}\dist(c(x), \cZ_\ast) \leq \muf \dist(x, \cX_\ast) \leq h(c(x)) - h^\ast.\qedhere$$  
\end{proof}
\noindent While the proof of the lemma shows that $\muh$ exists and is at least $\muf/\lc$, it can be much larger, as outlined in the introduction.
See the survey~\cite{10.5555/3226650.3226819} for further discussions on sharpness.

Finally, in addition to sharpness, we assume that $h$ satisfies a one-sided Taylor-series like approximation condition:
\begin{assumption}[$(b_{\leq})$-regularity]\label{eq:b-regularity}
{\rm
The following estimate holds:
	\begin{equation}
	h(z) + \langle v, y - z\rangle - h^\ast \leq o(\|y -z \|) \qquad \text{as $y \stackrel{\cZ_\ast}{\rightarrow} c(x_\star)$, $z \stackrel{\cM}{\rightarrow} c(x_\star)$ with $v \in \partial h(z)$,}
		\end{equation}
where $o(\cdot)$ is any univariate function satisfying $\lim_{t \rightarrow 0} o(t)/t = 0$.		}
\end{assumption}
\noindent This condition was recently introduced in~\cite{davis2021subgradientactivemanifold} in the context saddle-point avoidance in nonsmooth optimization.
It holds automatically when $h$ is convex or weakly convex.
The property also holds automatically when $\cZ_\ast$ is isolated at $c(x_\star)$ and $h$ is \emph{semismooth}~\cite{Mifflin77}; semismoothness is a broad property and holds for any semialgebraic function $h$~\cite{BDL09}.\footnote{A function $h$ is semialgebraic if its graph is the union of intersections of sets cut out by finitely many polynomial inequalities.}
The reader may consult~\cite{davis2021subgradientactivemanifold} for further examples of ($b_{\leq}$) regularity.

With our assumptions in hand, the following lemma proves the existence of a neighborhood around $x_\star$ on which several important inequalities are satisfied. These inequalities will be used throughout the proof of our main theorem.
\begin{lem}\label{lem:assumption}
Define $\muc := \frac{1}{2}\sigma_{r}(\nabla c(x_\star))$, $\lj := 2\lip_{\nabla c}(x_\star)$, and 
$
L_h := 2\lip_h(c(x_\star)),
$
and $\lc := 2\lip_c(x_\star)$.
Then there exists $R, C > 0$ such that the following hold for all $x, x' \in \overline{B_{R}(x_\star)}$, $v \in \partial h(x)$, $z := c(x), z' := c(x')$, and $\hat z \in P_{\cZ_\ast}(z)$:
\begin{enumerate}
\item\label{item:lipschitz} $\|v\| \leq \lh$;
\item\label{item:lipschitz2} $h(z) - h^\ast \leq \lh\dist(z, \cZ_\ast)$;
\item \label{lem:assumption:semismooth}
$
h(z) - h^\ast + \dotp{v, \hat z - z} \leq \frac{\muh}{8}\dist(z, \cZ_\ast)
$;
\item\label{lem:assumption:sharp} 
$
h(z) - h^\ast \geq \muh  \dist(z, \cZ_\star)
$;
\item \label{lem:beta_smooth} $\|c(x') - (c(x) + \nabla c(x)(x' - x)\| \leq \frac{\lj}{2}\|x - x'\|^2$;
\item \label{lem:Lipschitzc} $\dist(c(x), \cZ_\ast) \leq \lc\dist(x, \cX_\ast)$;
\item \label{lem:assumption:singularvalue} 
$
\sigma_{r}(\nabla c(x)) \geq \muc$;
\item\label{eq:tangentspaceapprox}  
$
\|P_{\cT_{\cM}(z)^\perp}(z - z')\| \leq C\|z - z'\|^2
$.
\end{enumerate}
\end{lem}
\begin{proof}
Clearly, Items~\ref{item:lipschitz},~\ref{item:lipschitz2},~\ref{lem:assumption:semismooth},  \ref{lem:assumption:sharp}, \ref{lem:beta_smooth}, and~\ref{lem:Lipschitzc} hold for any sufficiently small $\delta$. 
Part \ref{lem:assumption:singularvalue} follows from the constant rank assumption.
In addition, Part~\ref{eq:tangentspaceapprox} holds for all small $\delta$ since $\cM$ is locally a $C^2$ manifold at $c(x_\star)$. 
\end{proof}

\section{Linear convergence with known optimal value}

In this section, we introduce and analyze the \emph{Gauss-Newton-Polyak} method, which is denoted by $\gnp(x_0, T)$. 
The method is depicted in Algorithm~\ref{alg:GNP}. It takes as input two parameters: an initial point $x_0$ and a total number of oracle calls $T$.
\begin{algorithm}[H]
\begin{algorithmic}[1]
\Require{Initialization $x_0 \in \RR^d$ and number of oracle calls $T$.}
\For{$t = 0, \cdots, T-1$:}
\State {\bf Choose:} $v_t \in \partial h(c(x_t))$;
\State {\bf Update:}
\begin{equation*}
x_{t+1} = x_t - \gamma_t  \nabla c(x_t)^\dag v_t \label{eq:GNP_xplus} \qquad \text{ where  $\frac{h(c(x_t)) - h^\ast}{2\|P_{\range(\nabla c(x_t))} v_t\|^2}\leq \gamma_t \leq  \frac{h(c(x_t)) - h^\ast}{\|P_{\range(\nabla c(x_t))}v_t\|^2}$}.
\end{equation*}
\EndFor
\State \Return $\argmin_{0 \leq t \leq T} \{h(c(x_t))\}$.
\end{algorithmic}
\caption{\gnp($x_0, T$)}
\label{alg:GNP}
\end{algorithm}

The following theorem shows that $\gnp$ linearly converges to a minimizer of $h \circ c$. To state the theorem, we need the following constants: 
\begin{align}\label{eq:constants_for_theorem}
c_1 := \left(1 - \frac{\muh^2}{2\lh^2}\right)^{1/2};&& c_2 :=\frac{8\lj \lh^2}{9\muc^2\muh^2} ; && c_3 = \frac{4 \lh}{3\muc\muh}.
\end{align}
and 
\begin{align}\label{eq:radius_for_theorem}
\delta := \min\left\{ \frac{R}{2}, \frac{\muh}{32\lh \lc C}, \frac{(1-c_1)}{2c_2\lc}, \min\left\{ R, \frac{\muh}{16\lh \lc C}\right\}\frac{(1-c_1)}{4c_3\lc}\right\} 
\end{align}
where $R$ is defined in Lemma~\ref{lem:assumption}.
\begin{thm}[Convergence of $\gnp$.]\label{thm:GNPconvergence}
Fix $\varepsilon > 0$. Suppose that 
$$
\|x_0 - x_\star\| \leq \delta \qquad \text{ and } \qquad 
T \geq \left\lceil \frac{\log(\dist(c(x_0), \cZ_\ast)  \lh /\varepsilon)}{\log(2/(1+c_1))}\right\rceil.
 $$ 
Let $\tilde x_T = \gnp(x_0, T)$ and let $x_0, \ldots, x_T$ denote the iterates $\gnp$. Then
$$
\dist(c(x_T), \cZ_\ast) \leq \varepsilon/\lh \qquad \text{ and } \qquad h(c(\tilde x_T)) - h^\ast \leq  \varepsilon  .
$$ 
\end{thm}

As claimed in the introduction, a quick calculation shows that $T$ need only satisfy 
$$
T = \Omega((\lh/\muh)^2\log(\dist(c(x_0), \cZ_\ast) \lh/\varepsilon))
$$
before the output of $\gnp$ is an $\varepsilon$-optimal solution.\footnote{Here we use the bound $1/\log(2/(1+(1-x)^{1/2})) \leq 4/x$, which holds for $x \in (0, 1)$.} We now turn to the proof.

\subsection{Proof of Theorem~\ref{thm:GNPconvergence}}

To begin the proof, we state the following sequence lemma, which allows us to reduce the proof of the theorem to a ``one-step improvement" analysis.

\begin{lem}
Fix $c_1 \in (0, 1)$ and $c_2, c_3, \eta > 0$ and consider nonnegative sequences $\{a_t\}$ and $\{b_t\}$ satisfying
\begin{align}\label{seq:condition}
\left\{\begin{aligned}
a_{t+1} &\leq  c_1a_t + c_2a_t^2 \\
b_{t+1} &\leq b_t + c_3 a_t
\end{aligned}\right\}\qquad \text{if $b_t \leq \eta$.}
\end{align}
Suppose $a_0 \leq \min\left\{(1-c_1)/2c_2, {\eta(1-c_1)}/{4c_3}\right\}$ and $b_0 \leq \eta/2$. Then $b_t \leq \eta$ for all $t \geq 0$. Moreover, for any $\varepsilon > 0$,  $a_t \leq \varepsilon$ when $t \geq \left\lceil {\log(a_0/\varepsilon)}/{\log(2/(1+c_1))}\right\rceil.$
\end{lem}
\begin{proof}
A simple inductive argument shows that for $t \geq 1$, we have $a_t \leq (1/2 + c_1/2)^ta_0$  and 
$$
b_{t} \leq b_0 + c_3a_0\sum_{i=0}^{t-1}  (1/2+c_1/2)^i \leq b_0 + \frac{2c_3a_0}{1-c_1} \leq \eta,
$$
as desired.
\end{proof}

We will apply this lemma to the sequences
\begin{align*}
a_t = \dist(c(x_t), \cZ_\ast) \qquad\text{and} \qquad b_t = \|x_t - x_\star\|,
\end{align*}
with radius
$
\eta := \min\left\{ R, \muh/(16\lh \lc C)\right\},
$
and constants $c_1, c_2,$ and $c_3$, just as in~\eqref{eq:constants_for_theorem}.
These constants guarantee the base case of the lemma is satisfied. Indeed, we have $b_0 \leq \delta \leq \eta/2$. In addition, we have 
$$
a_0 \leq \lc \delta \leq \min\left\{(1-c_1)/2c_2, {\eta(1-c_1)}/{4c_3}\right\}.
$$
Thus, if we can show that $a_t$ and $b_t$ satisfy Condition~\eqref{seq:condition}, the lemma and our choice of $T$ guarantee that 
$$
a_T \leq \varepsilon/\lh \qquad \text{ and } \qquad h( c(\tilde x_T)) - h^\ast \leq  h(c(x_T)) - h^\ast \leq \lh a_T \leq \varepsilon,
$$
proving the theorem

To that end, fix $t \geq 0$ and denote 
$$
x := x_t, \quad z := c(x_t), \quad x_+ := x_{t+1}, \quad z_+ := c(x_+), \text{ and } v := v_t.
$$
We will also denote 
$$
a := \dist(z, \cZ_\ast), \quad a_+:=  \dist(z_+, \cZ_\ast), \quad b:= \|x - x_\star\|, \qquad b_+ := \|x_+ - x_\star\|.
$$
We now turn to proving Condition~\eqref{seq:condition}.

First, we will show that the projected subgradient $P_{\cT_{\cM}(z)}v$ points towards $\cZ_\ast$.
\begin{lem}[Aiming]\label{lem:aiming}
If $b \leq \eta$, we have 
$$
\dotp{P_{\cT_\cM(z)} v, z - \hat z} \geq \frac{3}{4}\max\left\{h(z) - h^\ast, \muh a\right\} \qquad \text{ for all $\hat z \in P_{\cZ_\ast}(z)$.}
$$
Consequently, $\|P_{\cT_\cM(z)} v\| \geq \frac{3\muh}{4}$.
\end{lem}
\begin{proof}
By Parts~\ref{lem:assumption:semismooth} and~\ref{lem:assumption:sharp} of Lemma~\ref{lem:assumption}, we have
$$
 \dotp{v, z - \hat z} \geq h(z) - h^\ast - \frac{\muh}{8}a \geq\frac{7}{8} \max\left\{h(z) - h^\ast, \muh a\right\}.
$$
Therefore, by $\|v\| \leq \lh$ and the orthogonal decomposition $I = P_{\cT_\cM(z)} + P_{\cT_\cM^\perp(z)} $, we have
$$
 \dotp{P_{\cT_{\cM}(z)}v, z - \hat z} \geq \frac{7}{8} \max\left\{h(z) - h^\ast, \muh a\right\} - \|P_{\cT_\cM^\perp(z)}(z - \hat z)\|\lh.
$$
Next, by Part~\ref{eq:tangentspaceapprox} of Lemma~\ref{lem:assumption}, we have 
\begin{align*}
\|P_{\cT_\cM^\perp(z)}(z - \hat z)\| \leq  Ca^2 \leq (\lc C b) a \leq \frac{\muh}{8\lh} a \leq \frac{1}{8\lh} \max\left\{h(z) - h^\ast, \muh a\right\},
\end{align*}
where the second inequality follows from Part~\ref{lem:Lipschitzc} of Lemma~\ref{lem:assumption} and the inclusion $c(x_\star) \in \cZ_\ast$; the third follows from the bound $\lc Cb \leq 2\lc  C\eta \leq \muh/8\lh$; and the fourth follows from Part \ref{lem:assumption:sharp} of Lemma~\ref{lem:assumption}.
Therefore, we have the desired bound:
$$
 \dotp{P_{\cT_{\cM}(z)}v, z - \hat z} \geq \frac{7}{8} \max\left\{h(z) - h^\ast, \muh a\right\} - \|P_{\cT_\cM^\perp(z)}(z - \hat z)\|\lh \geq  \frac{3}{4}\max\left\{h(z) - h^\ast, \muh a\right\}.
$$
\end{proof}

The next lemma shows that the step length is bounded.
\begin{lem}[Step length]\label{step_length}
If $b \leq \eta$, we have 
$$
b_+ - b \leq \|x_+ - x\| \leq  c_3a.
$$
\end{lem}
\begin{proof}
We have
\begin{align*}
 b_+ - b \leq \|x_+ - x\| =  \gamma \|\nabla c(x)^\dag v\| \leq  \frac{ \gamma \|\nabla c(x) \nabla c(x)^\dag v\| }{\sigma_{r}(\nabla c(x))}\leq \frac{ \gamma\|P_{\cT_\cM(z)}v\|}{\muc}  \leq \frac{ (h(z) - h^\ast)}{\muc\|P_{\cT_\cM(z)}v\|}, 
 \end{align*}
 where the second inequality follows since $\nabla c(x)$ is rank $r$; the third inequality follows by Part~\ref{lem:assumption:singularvalue} of Lemma~\ref{lem:assumption} and the equality $\nabla c(x) \nabla c(x)^\dag = P_{\range(\nabla c(x))} = P_{\cT_{\cM}(z)}$; and the last inequality follows by definition of $\gamma.$
To complete the proof, we simply notice that
$$
\frac{(h(z) - h^\ast)}{\|P_{\cT_\cM(z)}v\|} \leq \frac{4\lh a}{3\muh},
$$
which follows from Part~\ref{item:lipschitz2} of Lemma~\ref{lem:assumption} and Lemma~\ref{lem:aiming}.
\end{proof}

The final lemma proves each step of the algorithm contracts up to quadratic error.
\begin{lem}[One step improvement]
If $b \leq \eta$, we have 
$$
a_+ \leq c_1a + c_2a^2.
$$\end{lem} 
\begin{proof}
Let $\hat z \in P_{\cZ_\star}(z)$. First, observe that
\begin{align*}
\|z - \gamma P_{\cT_{\cM}(z)} v - \hat z\|^2 &=  \|z - \hat z\|^2 - 2\gamma  \dotp{P_{\cT_{\cM}(z)} v,z - \hat z} + \gamma^2 \|P_{\cT_\cM(z)} v\|^2 \\
&\leq a^2 - \frac{3\gamma  (h(z) - h^\ast)}{2}  + \gamma^2 \|P_{\cT_\cM(z)} v\|^2 \\
&\leq a^2 - \frac{(h(z) - h^\ast)^2}{2\|P_{\cT_\cM(z)} v\|^2} \\
&\leq \left(1 - \frac{\muh^2}{2\lh^2}\right)a^2,
\end{align*}
where the first inequality follows from Lemma~\ref{lem:aiming}; the second follows by definition of $\gamma$; and the third follows from  Parts~\ref{lem:assumption:sharp} and~\ref{item:lipschitz} of Lemma~\ref{lem:assumption}.
Next, observe that by Part~\ref{lem:beta_smooth} of Lemma~\ref{lem:assumption}, we have
$$
\|z_+ - (z - \gamma P_{\cT}(v))\| = \|c(x_+) - (c(x) + \nabla c(x)(x_+ - x))\|  \leq \frac{\lj\|x_+ - x\|^2}{2} \leq \frac{\lj c_3^2a^2}{2}= c_2a^2, 
$$
where the final inequality follows from Lemma~\ref{step_length}.
Therefore, we have
\begin{align*}
a_+ \leq \|z_+ - \hat z\| &\leq\|z - \gamma P_{\cT_{\cM}(z)} v - \hat z\| + \|z - \gamma P_{\cT_{\cM}(z)} v - z_+\|  \leq c_1 a + c_2 a^2,
\end{align*}
as desired.
\end{proof}

This completes the proof of Theorem~\ref{thm:GNPconvergence}.

\section{Linear convergence with a lower bound on $h^\ast$}

In this section, we show how to adapt $\gnp$ to the case where the optimal value $h^\ast$ is unknown by using a restart strategy. The approach that we take is identical to that developed in~\cite{hazan2019revisiting}.
The method, which we call \emph{Restarted-Gauss-Newton-Polyak}, is denoted by $\rgnp(x_0, h_0, T, K)$ and appears in Algorithm~\ref{alg:RGNP}. It takes as input four parameters: an initial point $x_0$, an initial lower bound 
$
h_0 < h^\ast,
$
a total number of oracle calls $T$ for the inner loop, and a total number of restarts $K > 0.$

\begin{algorithm}[H]
\begin{algorithmic}[1]
\Require{Initialization $x_0 \in \RR^d$, error tolerance $\epsilon \in \RR_+$, number of oracle calls $T$, number of restarts $K$, lower bound $h_0 < h^\ast$.}
\For{$k = 0, \cdots, K-1$:}
\State {\bf Set: } $x_{k,0} = x_0$.
\For{$t = 0, \cdots, T-1$:}
\State {\bf Choose:} $v_{k,t} \in \partial h(c(x_{k,t}))$;
\State {\bf Update:}
\begin{equation*}
x_{k,t+1} = x_{k,t} - \gamma_{k,t}  \nabla c(x_{k,t})^\dag v_{k,t} \label{eq:GNP_xplus} \qquad \text{ where  $\gamma_{k,t} :=  \frac{h(c(x_{k,t})) - h_k}{2\|P_{\range(\nabla c(x_{k,t}))} v_{k,t}\|^2}$}.
\end{equation*}
\EndFor
\State {\bf Update: } 
$$y_{k+1} = \argmin_{0 \leq t \leq T} \{h(c(x_{k,t}))\} \qquad \text{ and } \qquad  h_{k+1} = \frac{ h_k + h(c(y_{k+1}))}{2}.$$
\EndFor
\State \Return $y_K$
\end{algorithmic}
\caption{\rgnp($x_0, h_0, T, K$)}
\label{alg:RGNP}
\end{algorithm}

We now give some intuition for the method. 
First, let us define the ``ideal stepsize" by 
\begin{align}\label{eq:idealstep}
\gamma_{t, k}^\ast := \frac{h(c(x_{k,t})) - h^\ast}{\|P_{\range(\nabla c(x_{k,t}))} v_{k,t}\|^2}.
\end{align}
We consider two possibilities. First, if for some $k \leq K-1$, we have $\gamma_{k,t}^\ast/2 \leq \gamma_{k,t} \leq \gamma_{k,t}^\ast$ for all $t \leq T-1$, one may view $x_{k,0}, \ldots, x_{k,T}$ as a realization of the method \gnp($x_0, T$). Hence, Theorem~\ref{thm:GNPconvergence} shows that if $T$ is large enough, iterate $y_{k+1}$ will satisfy $h(c(y_{k+1})) - h^\ast \leq \varepsilon$.
On the other hand, if for some $k, t > 0$, we have  $ \gamma_{k,t} \geq \gamma_{k,t}^\ast$ the proof will show that $h^\ast - h_{k+1} \leq (h^\ast - h_{k})/2$. 
Thus every step of the method either improves our estimate of $h^\ast$ or is sufficiently near a solution.

The following theorem formalizes the above argument.
\begin{thm}\label{thm:RGNPconvergence}
Fix $\varepsilon > 0$. Define $c_1, c_2, c_3$, and $\delta$ as in equations~\eqref{eq:constants_for_theorem} and~\eqref{eq:radius_for_theorem}.
Suppose that
$
\|x_0 - x_\star\| \leq \delta,
$
$$
T \geq \left\lceil \frac{\log(\dist(c(x_0), \cZ_\ast) \lh /\varepsilon)}{\log(2/(1+c_1))}\right\rceil, \qquad\text{ and } \qquad 
K \geq 1+ \left\lceil \frac{\log((h^\ast - h_0)/\varepsilon)}{\log(2)}\right\rceil.$$ Then the point $y_K = \rgnp(x_0, h_0,  T, K)$ satisfies 
$$
h(c(y_K)) - h^\ast \leq \varepsilon. 
$$ 
\end{thm}
\begin{proof}
The proof follows the strategy of~\cite{hazan2019revisiting}. First suppose that $h_k \leq h^\ast$ for all $k \leq K$. Then for all $k \leq K-1$, we have
$$
0 \leq h^\ast - h_{k+1} \leq h^\ast - \frac{h(c(y_k)) + h_k}{2} \leq \frac{h^\ast - h_k}{2}.
$$
Consequently, 
$$
h(c(y_{K-1})) = 2h_K- h_{K-1} \leq h^\ast + (h^\ast - h_{K-1}) \leq h^\ast + 2^{-(K-1)}(h^\ast - h_0) \leq h^\ast + \varepsilon, 
$$
as desired.

Next suppose that $1 \leq k \leq K-1$ is the first index such that $h_k > h^\ast$. 
In this case, we must have $\gamma_{k-1,t} \leq \gamma_{k-1,t}^\ast$ for all $t \leq T-1$.
Indeed, if there exists $t^\ast$ with $\gamma_{k-1,t^\ast} \geq \gamma_{k-1,t^\ast}^\ast$, we have 
$$
\frac{h(c(x_{k-1, t^\ast})) - h_{k-1}}{2} \geq h(c(x_{k-1, t^\ast})) - h^\ast \implies h^\ast \geq \frac{h(c(x_{k-1, t^\ast})) + h_{k-1}}{2} = h_k, 
$$
a contradiction. 
Therefore, we have the bounds 
$\gamma_{k-1, t}^\ast \geq \gamma_{k-1,t} \geq \gamma_{k-1,t}^\ast/2$ for all $t \leq T-1$. Consequently, the point $y_{k}$ can be realized as the result of $\gnp(x_0, T)$. Therefore, by Theorem~\ref{thm:GNPconvergence}, $h(c(y_{K})) - h^\ast\leq h(c(y_{k})) - h^\ast \leq \varepsilon,$ as desired.
\end{proof}

\section{Numerical illustration on low-rank tensor sensing}

In this section, we perform a brief numerical illustration. 
We focus on low-rank tensor recovery problems of the form
\begin{align}\label{eq:tensor}
\min_{X \in \RR^{d \times r}} h(c(X)) := \|\cA(c(X)) - b\|_1,
\end{align}
where $r\leq d$ denotes the ``rank" of the problem; $c\colon \RR^{d\times r} \rightarrow \RR^{d^n}$ is a smooth mapping into the space of $n$-th order tensors; $\cA\colon \RR^{d^n}\rightarrow \RR^m$ is a linear operator; $b = \cA(c(X_\star)) + \epsilon$ is a (noisy) measurement of a fixed full-rank matrix $X_\star \in \RR^{d \times r}$ with random  $\epsilon \in \RR^m$; and $h\colon \RR^{d^n} \rightarrow \RR$ is the fixed penalty function $h(M) = \|\cA(M) - b\|_1$.
More specifically, throughout this section, $c$ takes the form:
$$
c(X) = \sum_{i=1}^r x_i^{\otimes n},
$$
where $x_i$ denotes the $i$th column of $X \in \RR^{d\times r}$.
Next, the linear mapping $\cA \colon \RR^{d^n} \rightarrow \RR^m$ is constructed from $2m$ Gaussian vectors $p_i, q_i \sim N(0, I_{d})$; on any tensor $M$ it returns 
$$
\cA(M) := (M(p_i^{\otimes n}) -  M(q_i^{\otimes n}))_{i=1}^m ,
$$
where $M(p_i^{\otimes n})$ denotes the evaluation of $M$ on $p_i^{\otimes n}$, etc. 
Finally, we choose noise vector $\epsilon$ in the following way: First we fix a ``failure probability" $\pfail \in [0, 1/2)$. Then for each $i$, we independently sample $\epsilon_i$ from a standard Gaussian with probability $\pfail$ or set $\epsilon_i$ to $0$ probability $1-\pfail$.

The motivation for measurement operator $\cA$ arises from the $n = 2$ case, which corresponds to a low-rank matrix recovery problem with \emph{centered quadratic sensing operator.}
This problem was first studied via a convex relaxation method in~\cite{chen2015exact}, 
where it was shown that the $\ell_1$ penalty recovers with optimal sample complexity.
Later, the Burer-Monteiro~\cite{BM03} formulation~\eqref{eq:tensor} was studied in~\cite{charisopoulos2021low}, where
it was shown that the subgradient method with Polyak stepsize ($\polyak$) method, which iterates 
$$
X_{t+1} := X_t - \gamma_t W_t \qquad  \text{ where $W_t \in \partial f(X_t)$}, 
$$ 
achieved the best known sample and computational complexity, up to a polynomial dependence on the condition number of $X_\star$, denoted $\kappa(X_\star)$.
Finally in~\cite{tong2021low}, the authors introduced a method with oracle complexity independent of $\kappa(X_\star)$. 
The method is called the ``Scaled subgradient method" ($\scaledsm$) and iterates
\begin{align}\label{eq:scaledsm}
X_{t+1} := X_t - \gamma_t W_t(X_t^TX_t)^{-1} \qquad \text{ where $W_t \in \partial f(X_t)$}.
\end{align}
Below we only compare against $\scaledsm$ in the $n=2$ case, since we are unsure how to define the proper adaptation for all higher order tensors.

We now comment on the validity of our main assumptions.
For general $n$, we note that $(b_{\leq})$-regularity (Assumption~\ref{eq:b-regularity}) is automatic. 
We conjecture the remaining assumptions -- sharpness and constant rank -- are valid with high probability when $m = \Omega( d n r )$ and $\pfail$ is a fixed constant.
We suspect this conjecture is true due to numerical evidence and prior work on the $n=2$ case~\cite{charisopoulos2021low}.
Indeed, sharpness of $h$ was shown in \cite[Lemma B.5]{charisopoulos2021low} under such assumptions.
In addition, we verify the constant rank property of~$c$ near $X_\star$ in Lemma~\ref{lem:bmconstantrank}.

\subsection{Implementation details.}

We now briefly comment on some implementation details:

\paragraph{Computing device and Language.} All of the experiments were programmed in PyTorch on a 2021 MacBook Pro with 10 CPU cores and 64GB of ram.

\paragraph{Initialization.}
In all of our experiments, we randomly initialize $X_0$ in a ball around $X_\star$ of radius $.1\|X_\star\|_F $.

\paragraph{Pseudoinverse calculation.} 
We numerically calculate the search direction $\nabla c(X_t)^\dag V_t$ via a conjugate gradient subroutine. 
We do this in two steps: First, we compute a subgradient of the composition $f := h \circ c$, which admits the form $\nabla c(X_t)^TV_t$ where $V_t \in \partial h(c(X))$. Then we find a minimal Frobenius norm solution $Z   \in \RR^{d\times r}$ of the linear system:
$$
\nabla c(X_t)^T\nabla c(X_t)Z = \nabla c(X_t)^TV_t,
$$
which is precisely $\nabla c(X_t)^\dag V_t.$ To formulate this system, that we use the identity: 
$$
\nabla c(X_t)^T\nabla c(X_t)Z = n(n-1)X((X^TX)^{\odot (n-2)} \odot Z^TX) + n Z(X^TX)^{\odot (n-1)},
$$
where $\odot$ denotes the Hadamard (element-wise) product.

\paragraph{Stepsize calculation.} 
After computing $\nabla c(X_t)^\dag X_t$, we compute the Polyak stepsize 
$$
\gamma_t := \frac{h(c(X_t)) - h^\ast}{\|P_{\range(\nabla c(X_t))}V_t\|^2}
$$
with the formula 
$$
\|P_{\range(\nabla c(X_t))}V_t\|^2 = \dotp{(\nabla c(X_t)^T\nabla c(X_t))\nabla c(X_t)^\dag v_t, \nabla c(X_t)^\dag v_t}.
$$
Importantly, this formula does not require us to form a matrix in $\RR^{d^n}.$

We now turn to the experiments.

\subsection{Dependence on conditioning of $X_\star$}\label{sec:conditionn2}

In Figures~\ref{fig:qs_exp1_GNP_SM} and~\ref{fig:qs_exp1_GNP_SM2}, we vary the condition number $\kappa := \kappa(X_\star)$, while fixing $n = 2$, $r=5$, $\pfail = 0$, $d = 1000$, and $m = 8dr$. We compare $\gnp$ to the $\polyak$ and $\scaledsm$. We observe that $\gnp$ performs comparably to $\scaledsm$ in both time and oracle complexity; see Figure~\ref{fig:qs_exp1_GNP_SM2}. 
On the other hand, as expected, both methods dramatically outperform $\polyak$ when $\kappa > 1$; see~\ref{fig:qs_exp1_GNP_SM};.

\begin{figure}[H]
    \centering
        \begin{minipage}{0.47\textwidth}
        \centering
        \includegraphics[width=\linewidth]{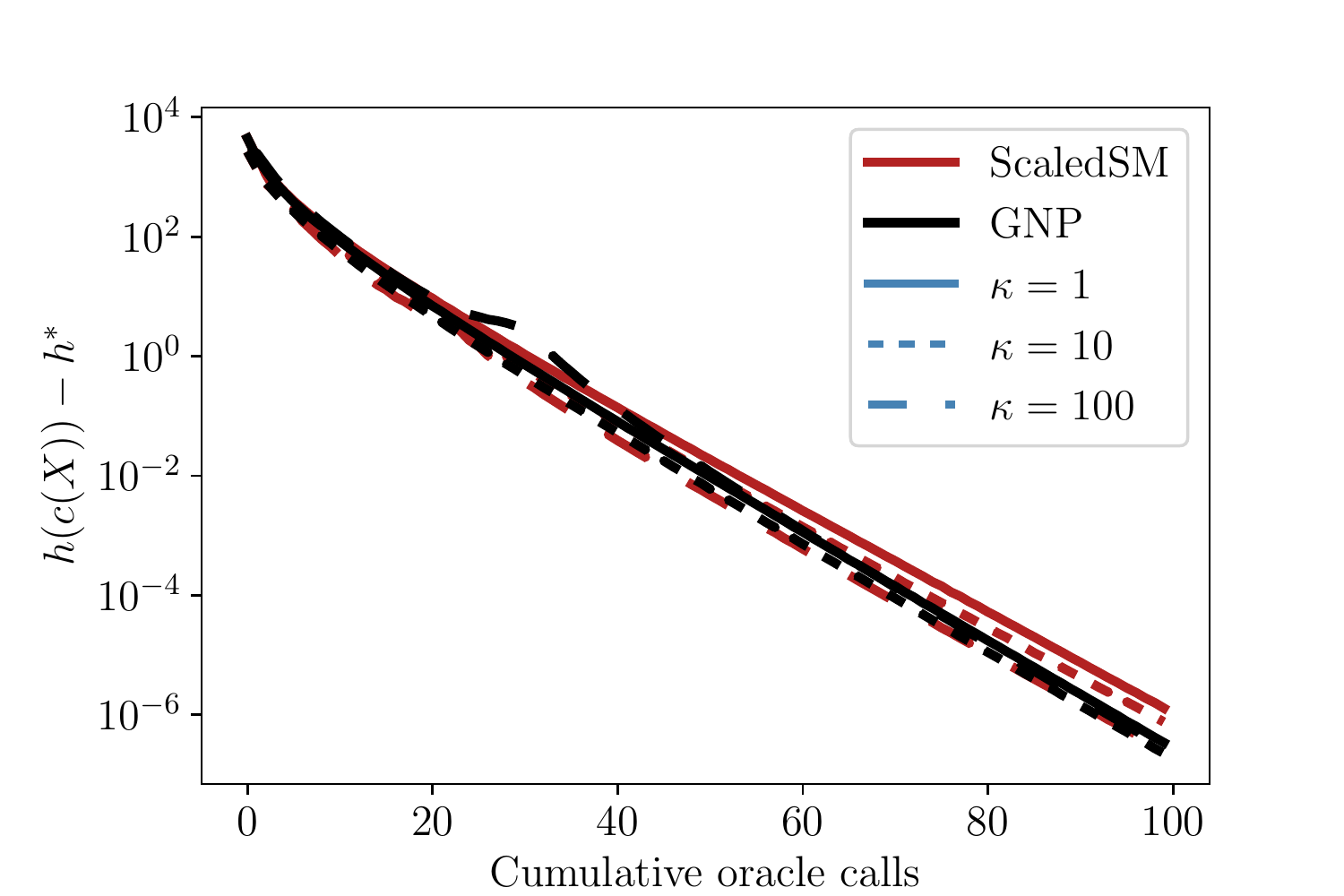}
    \end{minipage}%
    \begin{minipage}{0.47\textwidth}
        \centering
        \includegraphics[width=\linewidth]{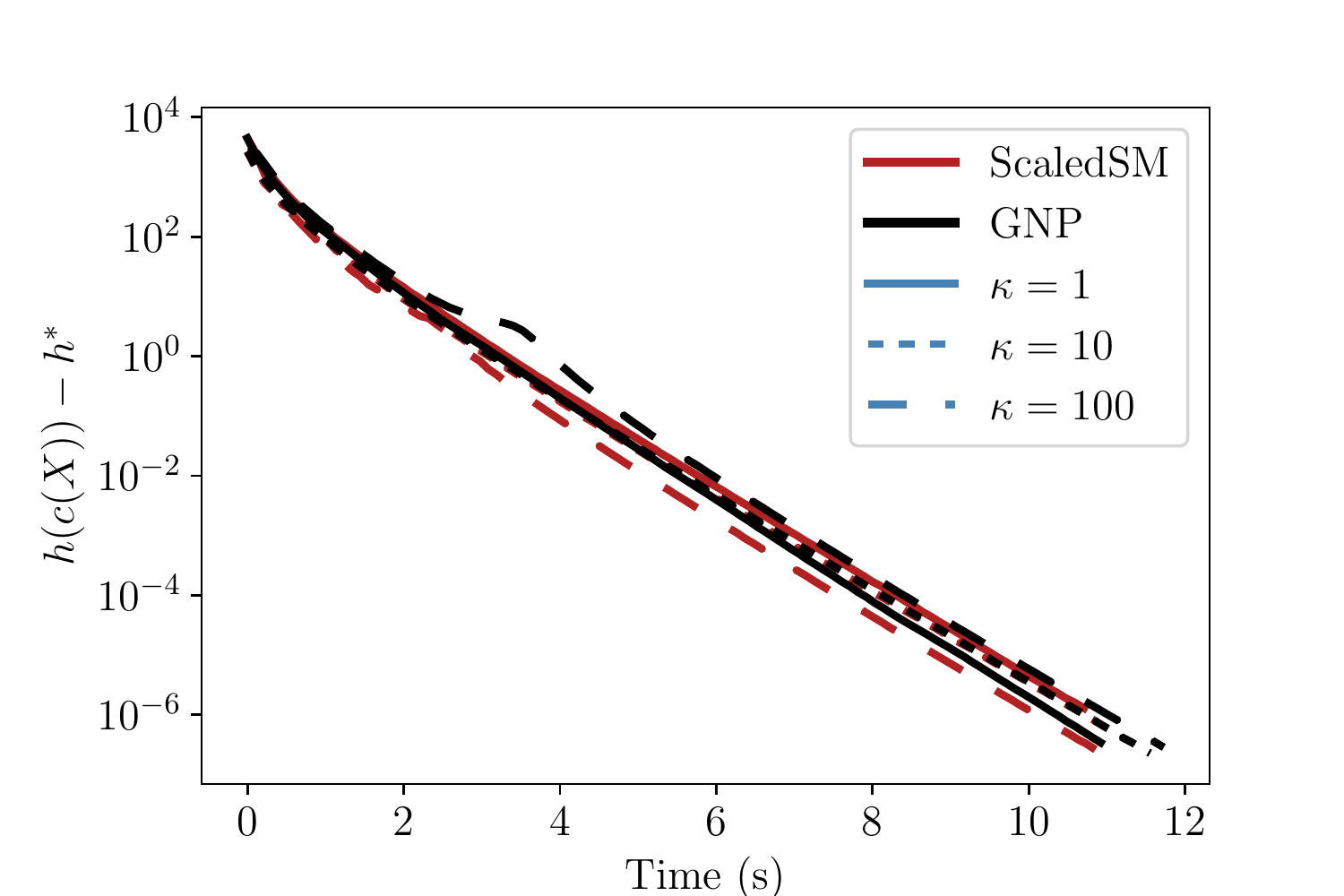}
    \end{minipage}
		\caption{Objective errors for $\scaledsm$, and $\gnp$ with varying condition number. See Section~\ref{sec:conditionn2} for details.}
		\label{fig:qs_exp1_GNP_SM2}
\end{figure}

\subsection{Dependence on order of tensor}\label{sec:orderoftensor}

In Figure~\ref{figs:ordertensor}, we vary the order $n$ of the tensor. 
We perform two sets of experiments. 
In both experiments, we set $r=5$, $\pfail = 0$, $m = 8ndr$.
In the top row, we set $\kappa = 1$ and $d=1000$; 
we set the max run time to 100 seconds.
In the bottom row, we set $\kappa = 3$ and $d = 100$; 
we set the maximal number of oracle calls for each run to be $10000$.
We see that in all cases, $\gnp$ outperforms $\polyak$. 
As expected, the performance gap is more dramatic when $\kappa > 1$.

\begin{figure}[H]
    \centering
    \begin{minipage}{0.47\textwidth}
        \centering
        \includegraphics[width=\linewidth]{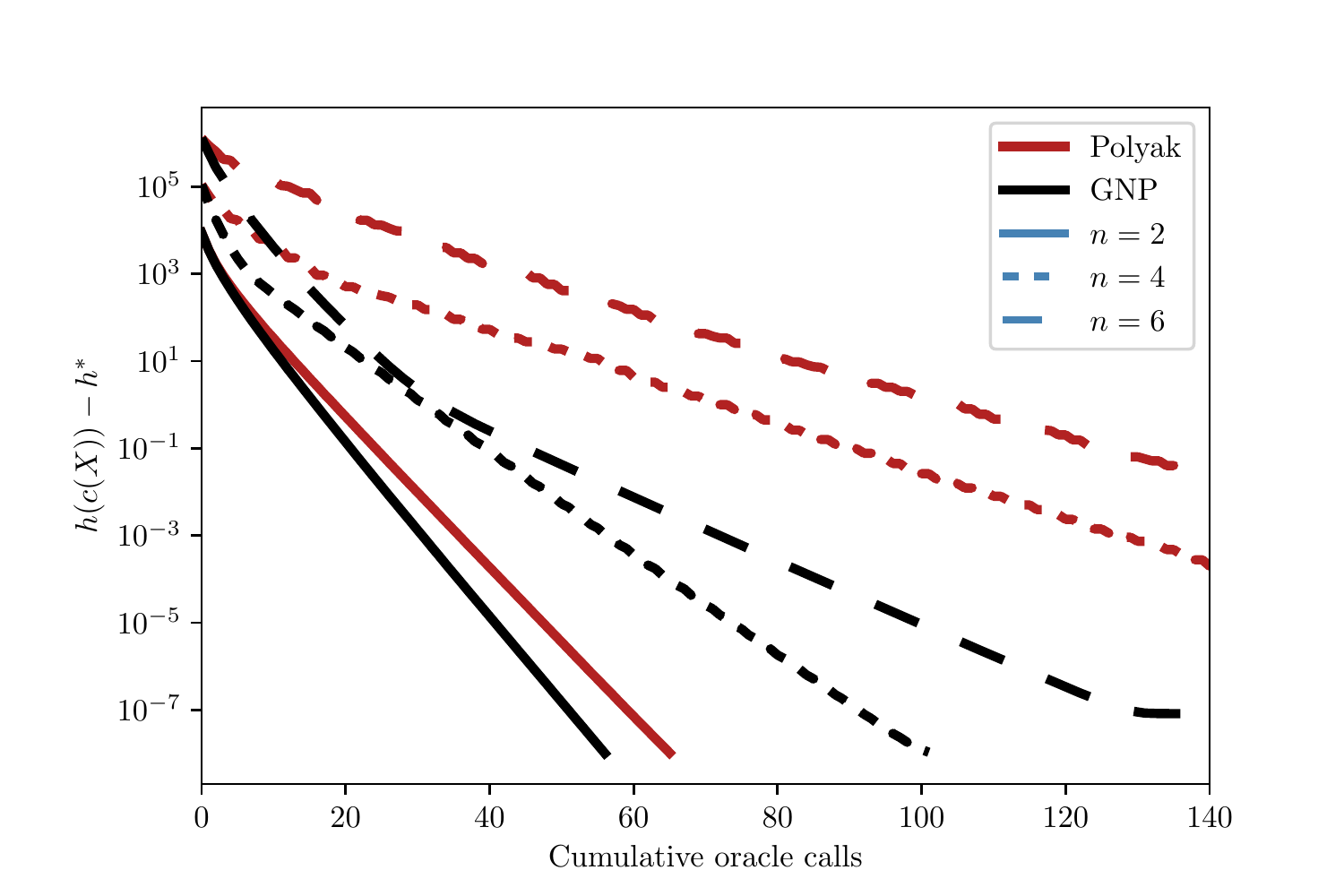}
    \end{minipage}%
    \begin{minipage}{0.47\textwidth}
        \centering
        \includegraphics[width=\linewidth]{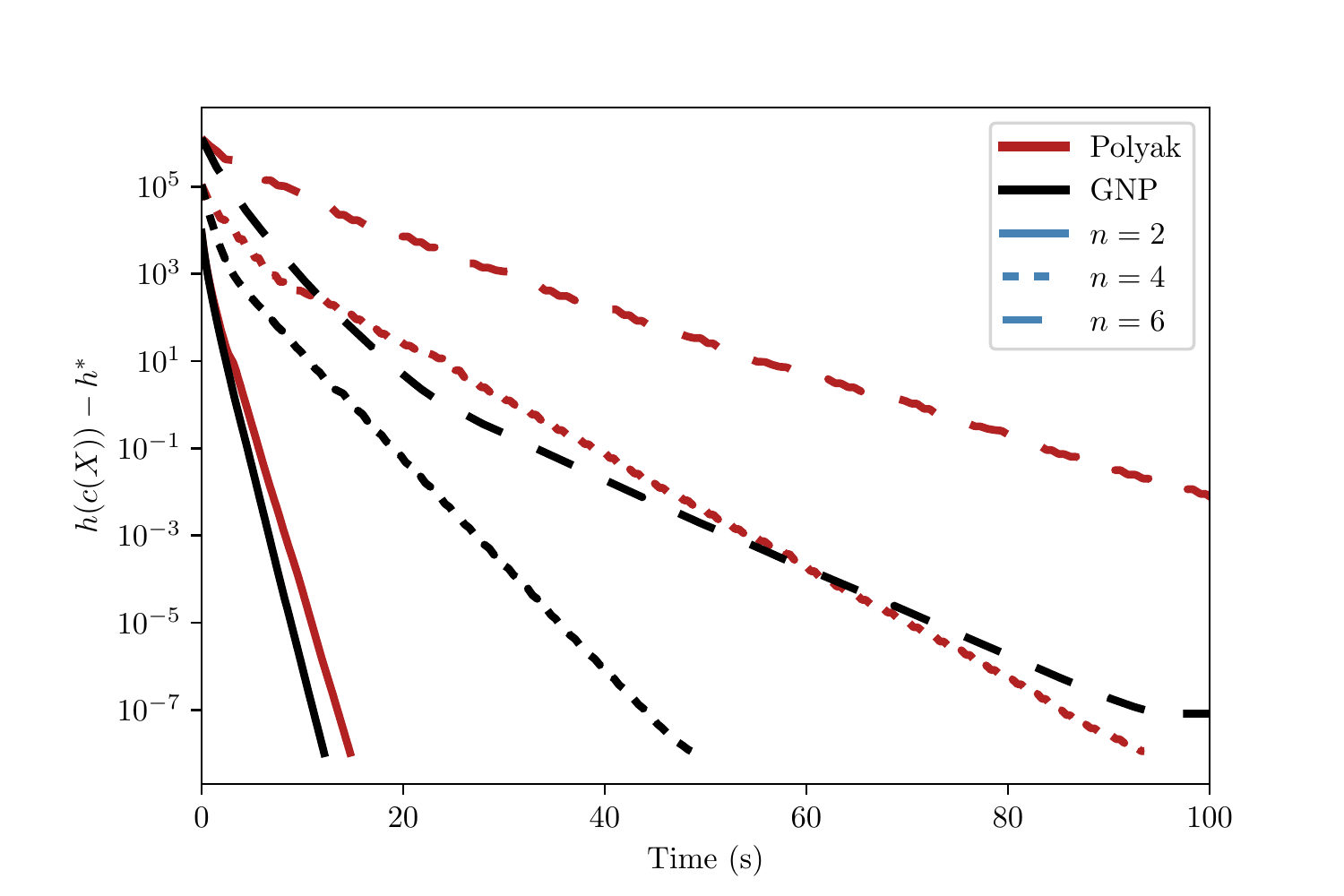}
    \end{minipage}
    \begin{minipage}{0.47\textwidth}
        \centering
        \includegraphics[width=\linewidth]{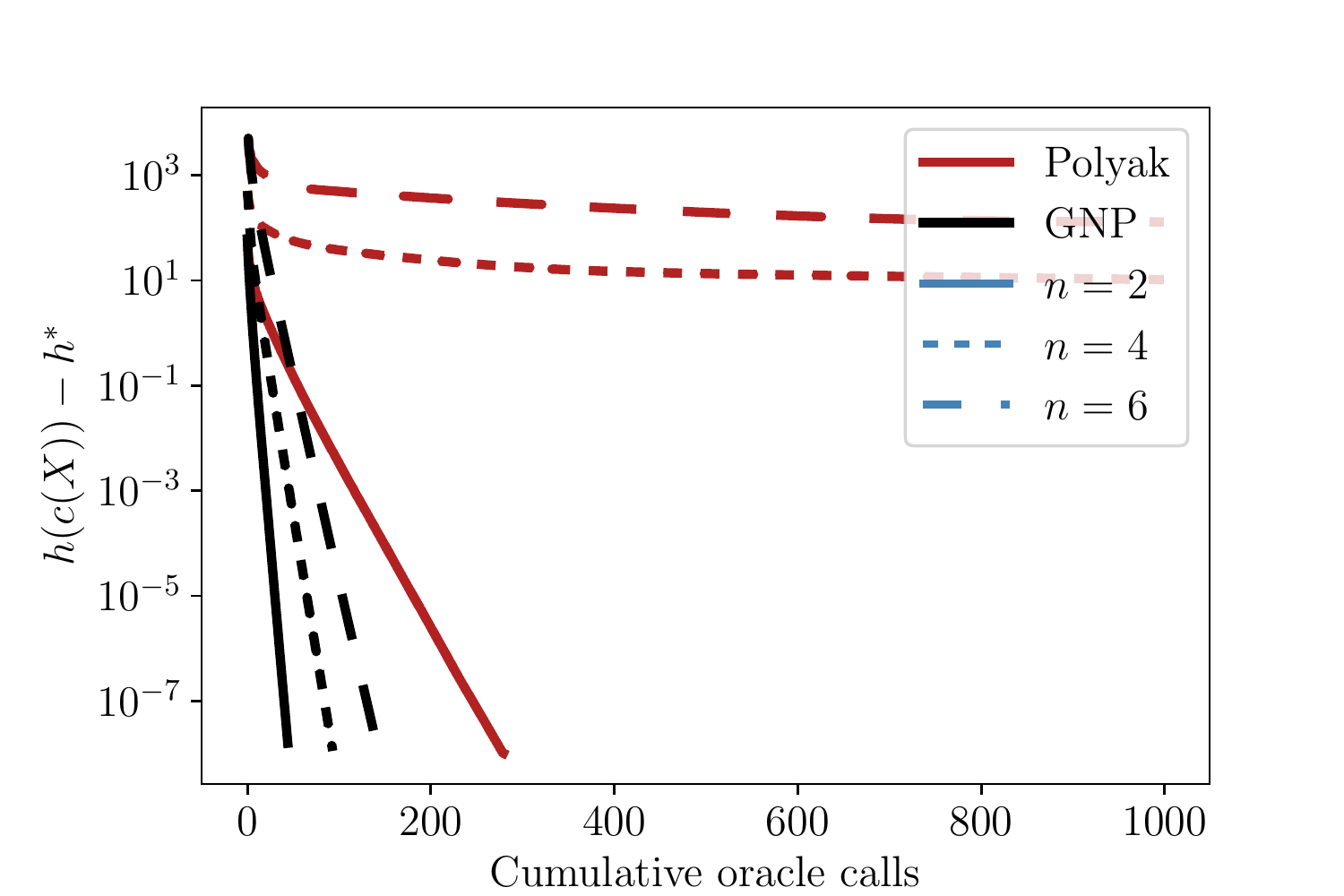}
    \end{minipage}%
    \begin{minipage}{0.47\textwidth}
        \centering
        \includegraphics[width=\linewidth]{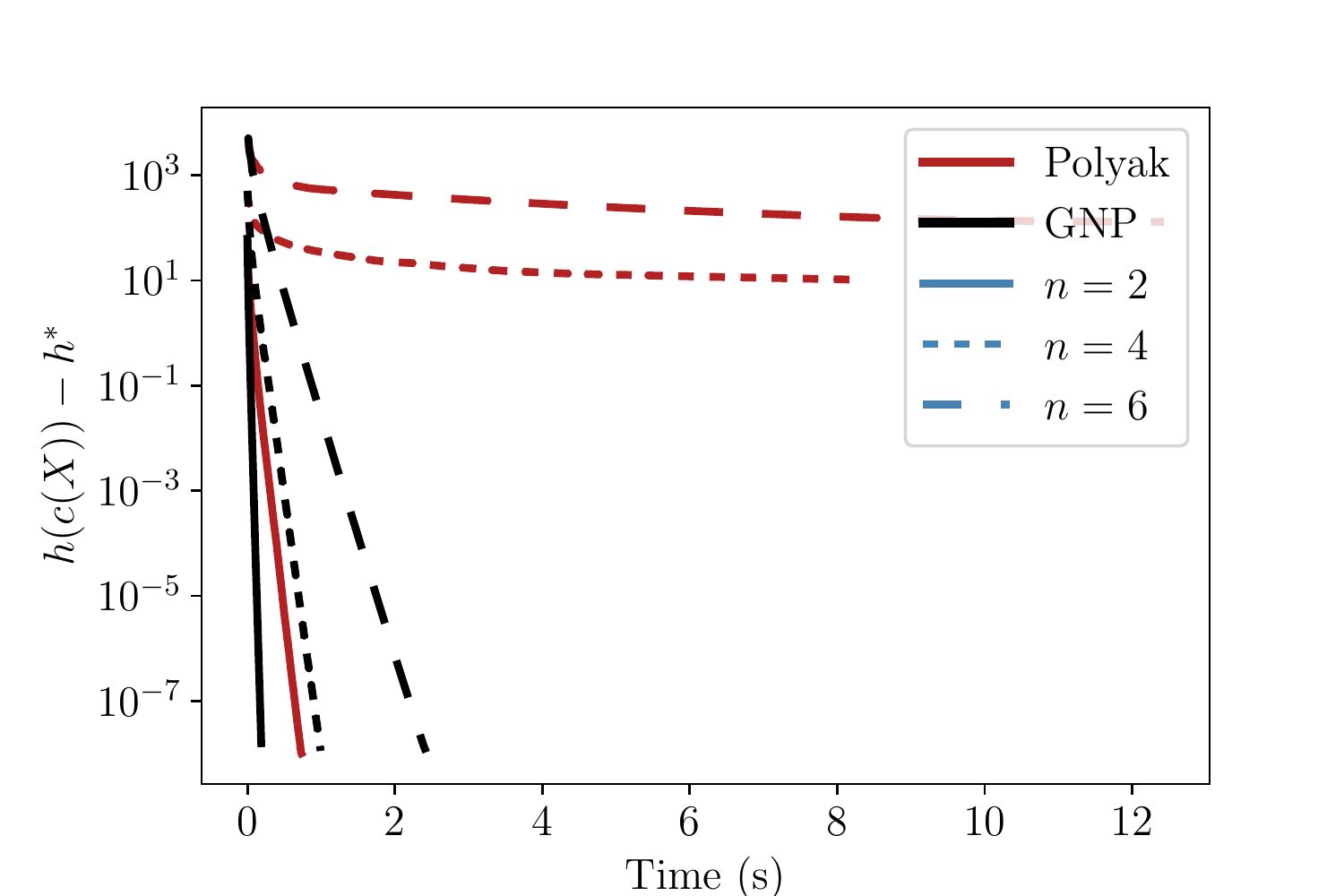}
    \end{minipage}    
		\caption{Objective errors for $\polyak$ and $\gnp$ with varying tensor order $n$. Top row: case $\kappa(X_\star) = 1$, $d = 1000$, $r=5$. Bottom row: case $\kappa(X_\star) = 3$, $d=100$, $r=5$. See Section~\ref{sec:orderoftensor} for details.}		\label{figs:ordertensor}
\end{figure}

\subsection{Restarts and unknown optimal value}\label{sec:rankmatrixsensing}

In Figure~\ref{fig:restarts}, we vary $\pfail$, while fixing $\kappa(X_\star) = 5$, $d=100$, $n=2$, $r=5$, and $m = 8dr$. 
In this case, the optimal value $h^\ast$ is unknown, so we instead we apply $\rgnp$ and a ``restarted $\polyak$" method, denoted $\rpolyak$, with an initial guess of $h_0 = 0$. 
In addition, we set the total inner loop size $T$ and the number of restarts $K$ differently for both methods.
For $\rgnp$, we $T = 200$ and $K = 50$.
For $\rpolyak$, we set $T=1000$ and $K=10$. 
Our rationale for these numbers is the following: 
First, we would like to cap the total number of oracle calls by $TK = 10000$. 
Second, we ensure $T$ is large enough that non-restarted methods $\gnp$ and $\polyak$ would perform well. 
Finally, we set $K$ so that it exhausts the remaining budget of oracle calls. 
We see that both methods eventually reach $10^{-8}$ objective error. 
In addition, $\gnp$ continues to outperform $\polyak$.

\begin{figure}[H]
    \centering
    \begin{minipage}{0.47\textwidth}
        \centering
        \includegraphics[width=\linewidth]{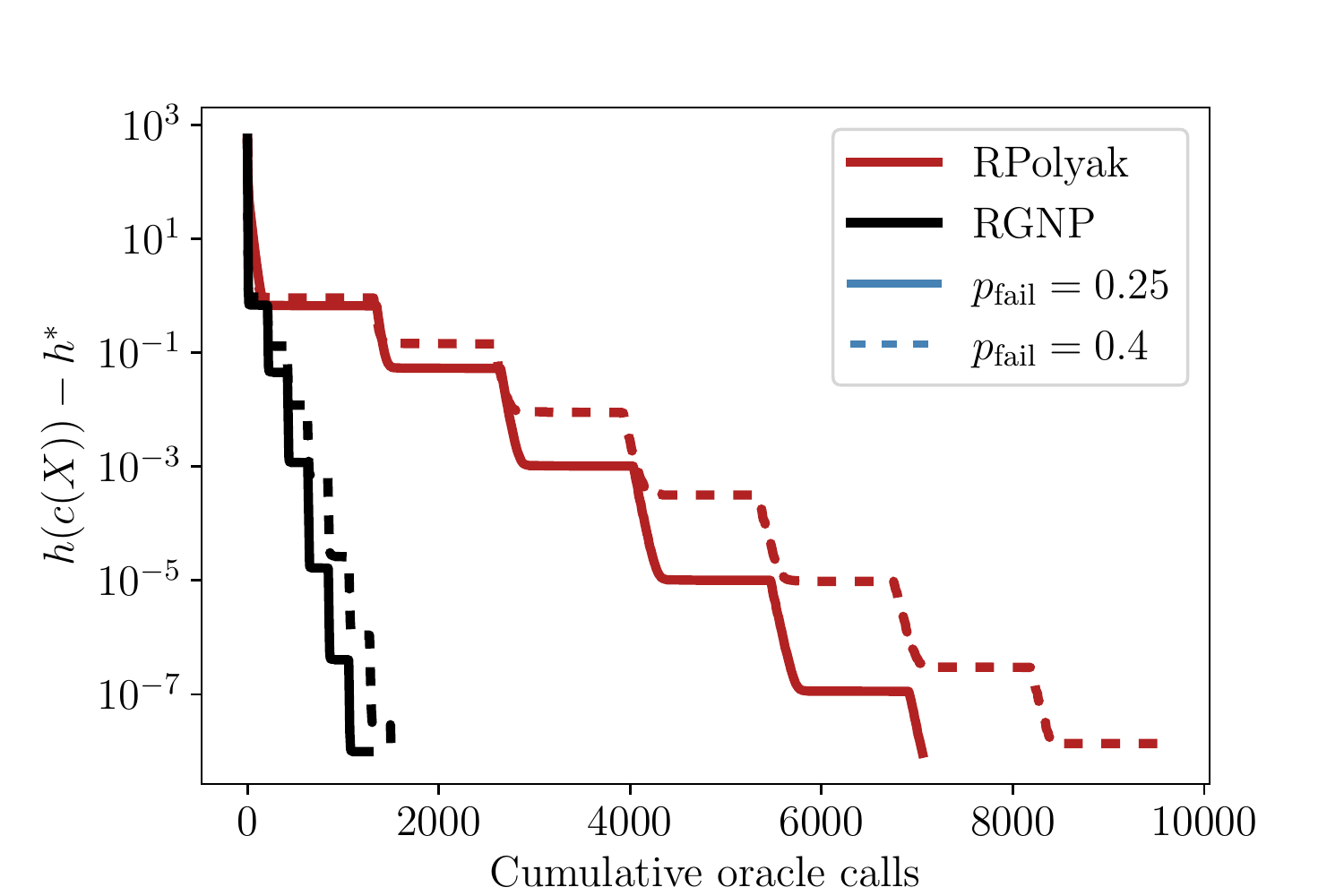}
    \end{minipage}%
    \begin{minipage}{0.47\textwidth}
        \centering
        \includegraphics[width=\linewidth]{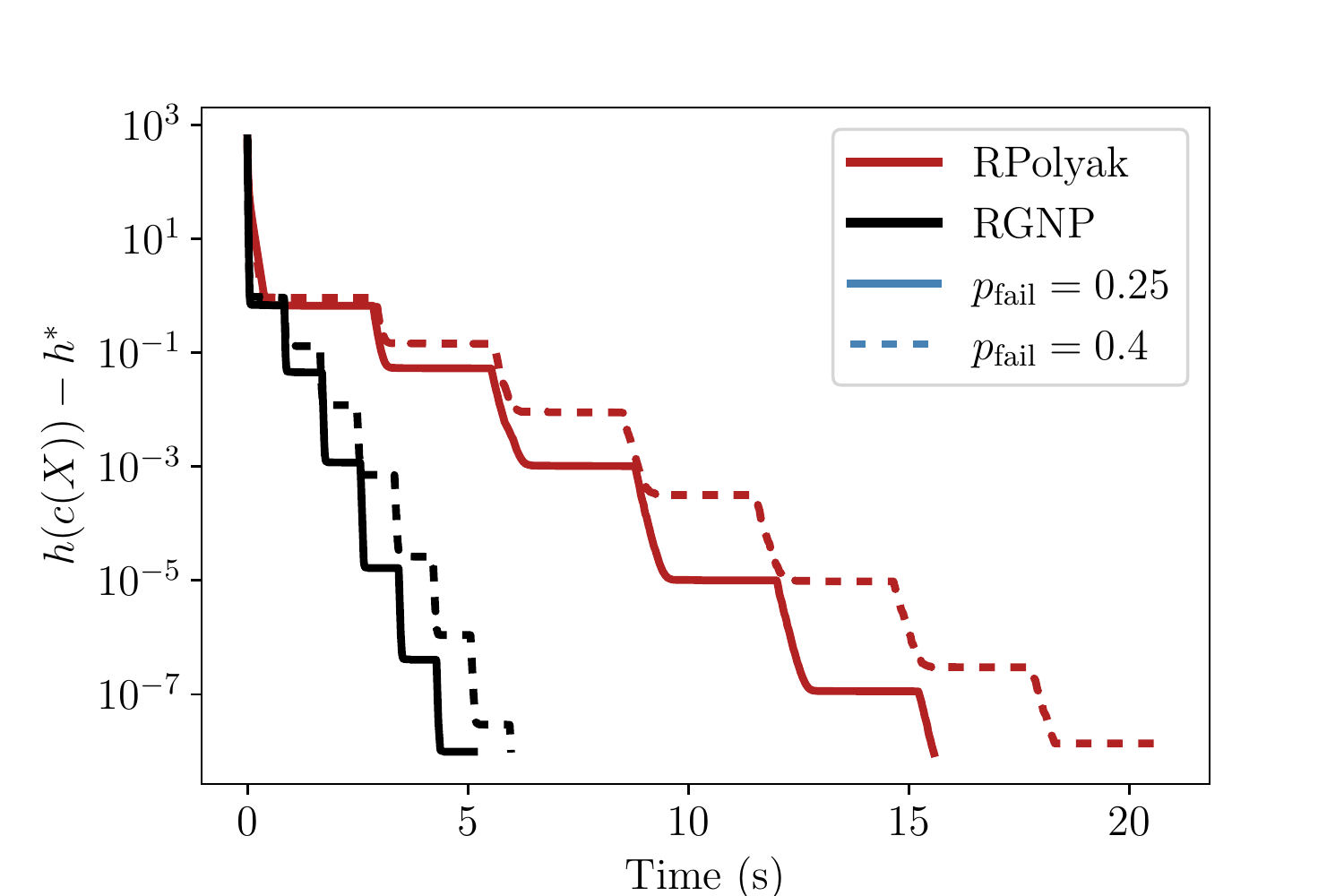}
    \end{minipage}
		\caption{Objective errors for $\rpolyak$ and $\rgnp$ with varying $p_{fail}$.}
		\label{fig:restarts}
\end{figure}

\subsection{Dependence on rank}\label{sec:rankmatrixsensing}

In Figure~\ref{fig:ranks}, we vary the rank $r$, while fixing $n = 3$, $\pfail = 0$, $d = 100$, and $m = 8ndr$, and $\kappa(X_\star) = 3$. We compare $\gnp$ to the $\polyak$. We observe that $\gnp$ outperforms $\polyak$ in both time and oracle complexity. 

\begin{figure}[H]
    \centering
    \begin{minipage}{0.47\textwidth}
        \centering
        \includegraphics[width=\linewidth]{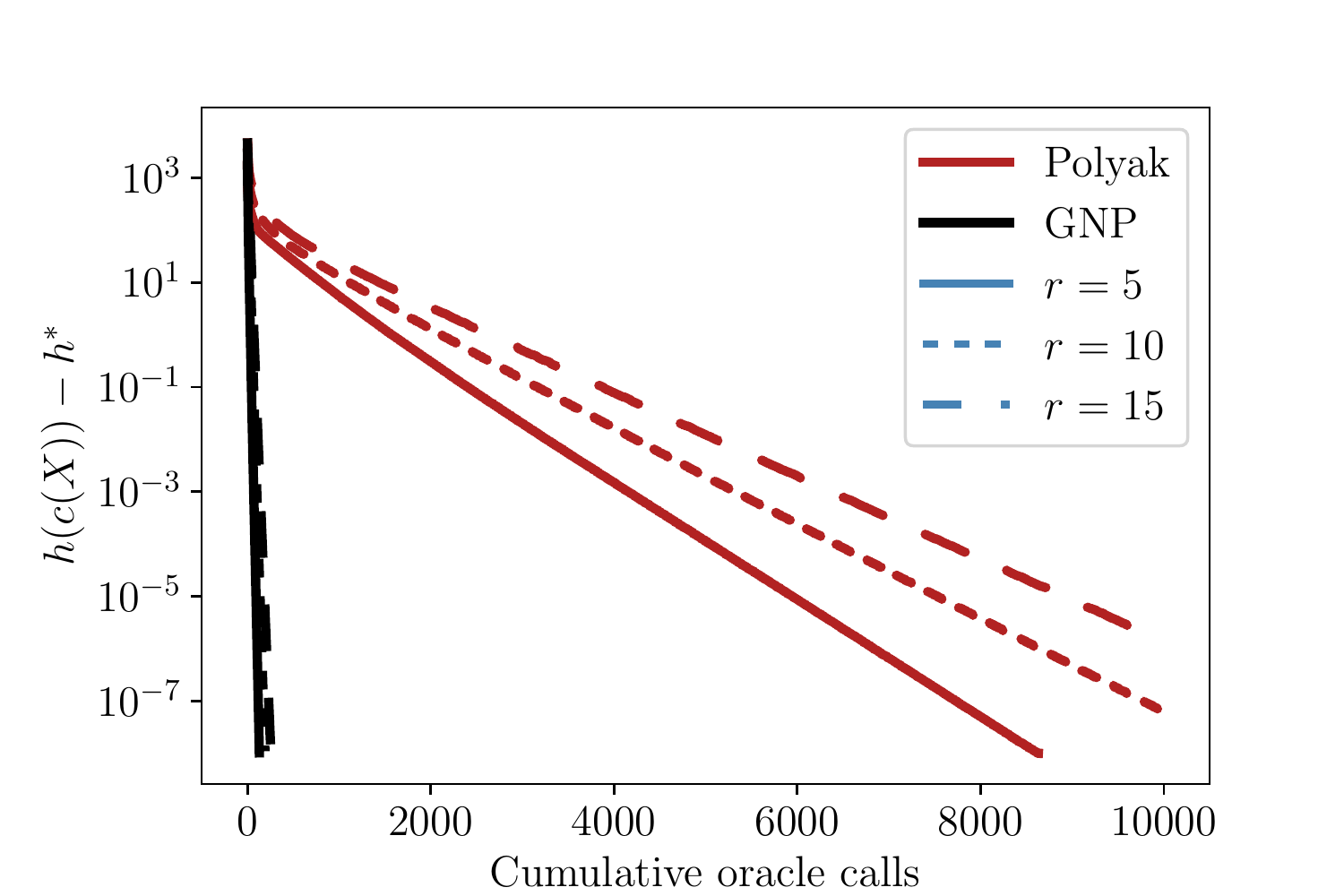}
    \end{minipage}%
    \begin{minipage}{0.47\textwidth}
        \centering
        \includegraphics[width=\linewidth]{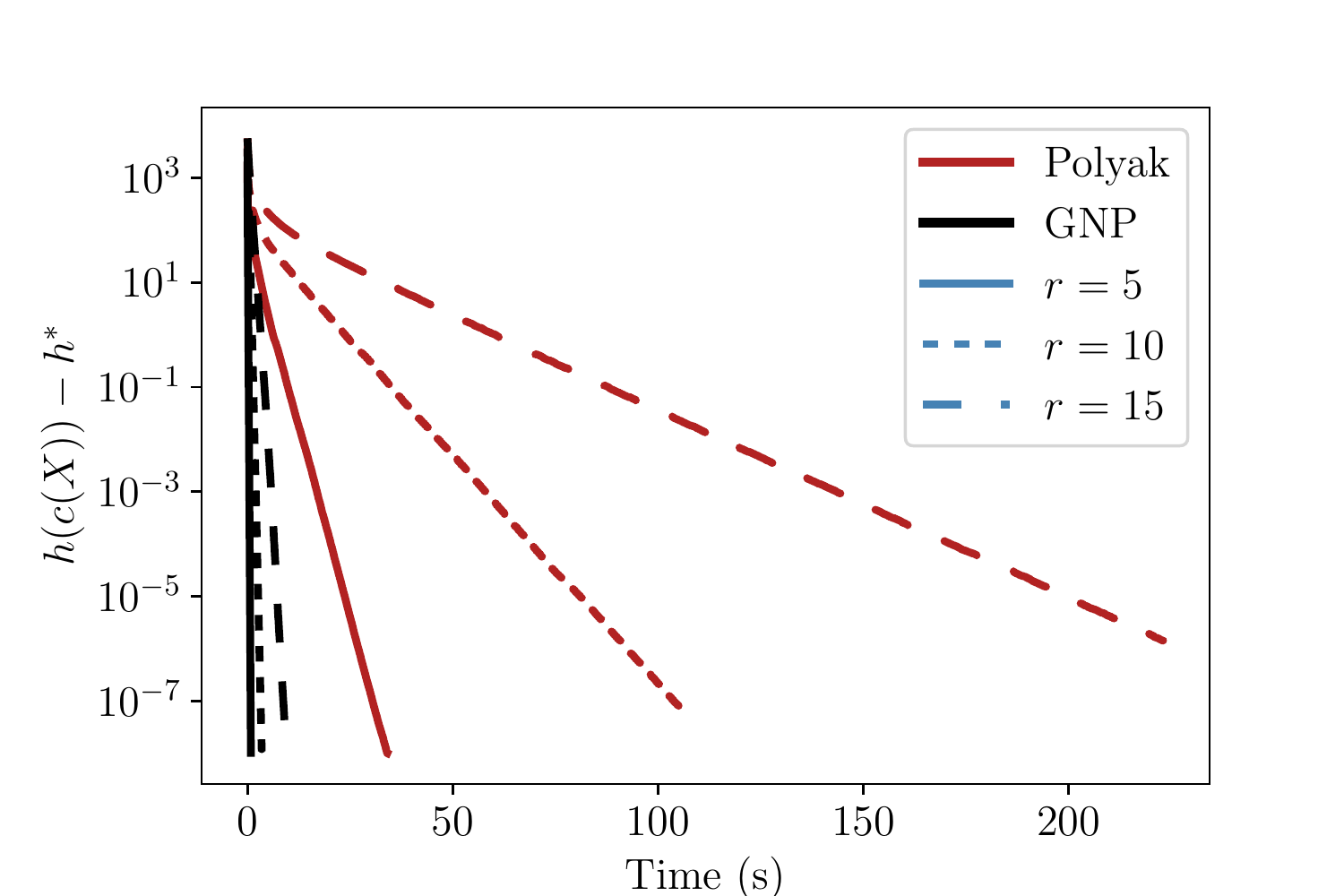}
    \end{minipage}
		\caption{Objective errors for $\polyak$ and $\gnp$ with varying rank.}
		\label{fig:ranks}
\end{figure}

\section*{Acknowledgement}
We thank Mateo D{\'i}az for pointing out that the restarted method of~\cite{hazan2019revisiting} generalizes beyond the convex setting.
We thank Dmitriy Drusvyatskiy for comments on this work.
We thank Nicolas Boumal for suggesting the reference~\cite[Section 4.1.3]{absil2009optimization}, which describes the retraction.

\appendix 

\section{Constant rank property of the Burer-Monteiro mapping}

\begin{lem}\label{lem:bmconstantrank}
Define the mapping Burer-Monteiro mapping $c \colon \RR^{d\times r} \rightarrow \RR^{d \times d}$ by 
$$
c(X) = XX^T \qquad \text{ for all $X \in \RR^{d\times r}.$}
$$
Suppose $X_{\star}\in \RR^{d \times r}$ is full column rank. Then $\nabla c(X)$ is rank $\frac{1}{2}r(r+1)$ near $X_\star$.
\end{lem}
\begin{proof}
Consider the Jacobian mapping $\nabla c(X)$, which acts on $Y \in \RR^{d\times r}$ as follows: 
$$
\nabla c(X)Y = XY^T + YX^T.
$$
To prove the result, we will show that the dimension of the kernel of this mapping is the quantity $nr - \frac{1}{2}r(r+1)$ at any full rank matrix $X$.
Since all matrices near $X_\star $ are full rank, this is all we need to show.
To that end, let $X = QR$ denote the QR factorization of $X$. Define 
$$
T_Q := \ker(\nabla c(X))R^{-1} = \{Y \in \RR^{d\times r} \colon QY^T + YQ^T = 0\}.
$$
Notice that $\dim\ker(\nabla c(X)) = \dim T_Q$.
Moreover, it is known that $\dim T_Q = nr - \frac{1}{2}r(r+1)$, since 
$T_Q$ is the tangent space to the so-called Steifel manifold $\cM = \{Q \in \RR^{d \times r} \colon Q^TQ = I\}$ at $Q$~\cite[Equation (7.16)]{boumal2020introduction}. This completes the proof.
\end{proof}

	\bibliographystyle{plain}
	\bibliography{bibliography}
\end{document}